\documentclass[11pt,reqno]{amsart}
\usepackage[T1]{fontenc}
\usepackage{lmodern,microtype}
\usepackage[margin=1.03in]{geometry}
\usepackage{amsmath,amssymb,mathtools,amsthm}
\usepackage{array,booktabs,longtable}
\usepackage{xcolor}
\usepackage[colorlinks=true,linkcolor=blue!50!black,
  citecolor=blue!50!black,urlcolor=blue!50!black]{hyperref}
\usepackage[capitalize,noabbrev]{cleveref}
\allowdisplaybreaks
\numberwithin{equation}{section}
\theoremstyle{plain}
\newtheorem{theorem}{Theorem}[section]
\newtheorem{lemma}[theorem]{Lemma}
\newtheorem{proposition}[theorem]{Proposition}
\DeclareMathOperator{\Frac}{Frac}
\DeclareMathOperator{\Spec}{Spec}
\DeclareMathOperator{\dist}{dist}
\newcommand{\QQ}{\mathbb Q}
\newcommand{\UU}{\mathcal U}
\newcommand{\mm}{\mathfrak m}

\title[Coherence and integral closure]{A one-dimensional coherent domain\\
with non-Pr\"ufer integral closure}

\author{Viet-Hoang Tran}
\address{Department of Mathematics, National University of Singapore, Singapore 119076}
\email{hoang.tranviet@u.nus.edu}
\urladdr{https://vh-tran.github.io/}

\author{Thieu N. Vo}
\address{Department of Computer Science, University of Bath, United Kingdom}
\email{ntv22@bath.ac.uk}

\author{Tan M. Nguyen}
\address{Department of Mathematics, National University of Singapore, Singapore 119076}
\email{tanmn@nus.edu.sg}
\urladdr{https://tanmnguyen89.github.io/}

\hypersetup{pdftitle={A one-dimensional coherent domain with non-Pruefer integral closure},
  pdfauthor={Viet-Hoang Tran}}

\begin{document}
\raggedbottom
\begin{abstract}
In a question recorded in 1978, Vasconcelos asked whether the integral
closure of a one-dimensional coherent local domain in its fraction field
is Pr\"ufer. We construct such a domain whose integral closure is not
Pr\"ufer, giving a negative answer.
\end{abstract}
\maketitle

\section{Introduction}\label{sec:introduction}

All rings are commutative with identity. A domain $D$ is
\emph{coherent} if every finitely generated ideal is finitely presented,
and \emph{Pr\"ufer} if every nonzero finitely generated ideal is
invertible. Equivalently, coherence requires finite generation of
\[
 \ker\bigl(D^r\xrightarrow{(a_1,\ldots,a_r)}D\bigr)
 \qquad(r\geq1,\ a_i\in D).
\]
Thus coherence controls relations among finite sets of generators,
whereas Noetherianity requires every ideal to be finitely generated.
Both Noetherian and Pr\"ufer domains are coherent
\cite[\S6]{GeroldingerKimLoper2026}.

The \emph{integral closure} of $D$ in $\Frac D$ is
\[
 D'=\{x\in\Frac D:f(x)=0\text{ for some monic }f\in D[T]\}.
\]
The domain $D$ is
\emph{integrally closed} (or \emph{normal}) if $D'=D$.

Two one-dimensional results explain the motivation:
\begin{equation}\label{eq:classical-motivation}
 \begin{aligned}
 D\text{ Noetherian},\quad\dim D=1
     &\ \Longrightarrow\ D'\text{ Dedekind},\\
 D\text{ coherent},\quad D=D',\quad\dim D=1
     &\ \Longrightarrow\ D\text{ Pr\"ufer}.
 \end{aligned}
\end{equation}
The first is the Krull--Akizuki theorem
\cite[Tag~09IG]{StacksProject}; the second appears in
\cite[Corollary~11]{PapickCoherentOverrings1979}.
Since Dedekind domains are Pr\"ufer, these implications suggest replacing
Noetherianity by coherence: must every one-dimensional coherent domain
have Pr\"ufer integral closure? The issue is
whether the finite relations in $D$ impose enough control on its
integral elements. Because $\dim D'=\dim D$ and Pr\"ufer domains are
coherent, the question is equivalently whether $D'$ is coherent
whenever $D$ is one-dimensional and coherent.

The local form was already recorded by Papick in 1978
\cite[\S3]{Papick1978}. Glaz--Vasconcelos formulated the assertion as
Conjecture~C$_3$ in their study of flat ideals
\cite[p.~223]{GlazVasconcelos1984}. It subsequently appeared as
Problem~65 in Chapman--Glaz's 2000 collection
\cite{ChapmanGlaz2000} and Problem~3 in the 2014 collection of
Cahen--Fontana--Frisch--Glaz \cite{CahenFontanaFrischGlaz2014}.
The survey of Geroldinger--Kim--Loper still lists it as Problem~5
in the version dated 28~February~2026
\cite[\S6]{GeroldingerKimLoper2026}.

The positive results impose extra control on the base, its
integral closure, or its overrings. Examples include very finite
extensions, finite generation of the integral closure, coherence of
overrings, and stability of ideals; \cref{sec:literature} gives the
precise hypotheses and their relationships. A complementary result
shows the severity of a possible failure: in a one-dimensional normal
local domain that is not Pr\"ufer, the intersection of any two
incomparable principal ideals is not finitely generated
\cite[Remark~2.2]{GuerrieriLoper2021}. The 2026 survey highlights this
obstruction and the scarcity of known coherent examples outside the
Noetherian and Pr\"ufer classes \cite[\S6]{GeroldingerKimLoper2026}.

We give a negative answer, already for a local domain.

\begin{theorem}\label{thm:main}
There exists a one-dimensional coherent local domain whose integral closure
is not Pr\"ufer.
\end{theorem}

The construction separates two kinds of bounds. In the plane-curve
domains $A_n=\QQ[[t^n,t^{n+1}]]$, relations on $r$ elements require
at most $r$ generators, independently of $n$, while the conductor
scale grows quadratically. The first bound makes
$A=\prod_{\UU}A_n$ coherent. At the second scale, two prime order
cuts $P\subsetneq Q$ satisfy
\[
 aP=P\quad(a\notin P),\qquad
 R=(A/P)_{Q/P},\qquad \Spec R=\{0,\mm\}.
\]
The divisibility identity lifts relations through $A\to A/P$;
localization then gives the coherent one-dimensional domain $R$
of \cref{sec:construction}.

In \cref{sec:obstruction}, irrationally spaced exponents yield a
fraction $z\in\Frac R$ for which
\[
 R[z]/\mm R[z]\cong(R/\mm)[T].
\]
Thus $\dim R[z]\geq2$. A one-dimensional Pr\"ufer integral closure
would force every overring of $R$ to have dimension at most one,
giving the contradiction. The proof detects the failure through a
single overring and does not require an explicit description of the
integral closure.

\section{The construction}\label{sec:construction}

Fix a nonprincipal ultrafilter $\UU$ on $\{2,3,\ldots\}$ and set
\[
 A_n=\QQ[[t^n,t^{n+1}]]\subseteq\QQ[[t]],\qquad
 A=\prod_{\UU}A_n.
\]
Here $A_n$ is the image of $\QQ[[X,Y]]$ under
$(X,Y)\mapsto(t^n,t^{n+1})$. Write ``almost everywhere'' for a set
belonging to $\UU$; thus $[a_n]=[b_n]$ if $a_n=b_n$ almost everywhere.
The ultrafilter alternative makes $A$ a domain. Let $\nu_n$ be the
$t$-adic order on $\QQ((t))$, with $\nu_n(0)=\infty$. We use
\[
 \nu_n(ab)=\nu_n(a)+\nu_n(b),\qquad
 \nu_n\!\left(\sum_i a_i\right)\geq\min_i\nu_n(a_i),
\]
with equality in the second formula when the minimum is unique.
Define two order cuts:
\begin{align}
 P&=\bigl\{[a_n]:(\forall C\in\mathbb Z_{\geq1})\,
                  \{n:\nu_n(a_n)>Cn^2\}\in\UU\bigr\},\label{eq:P}\\
 Q&=\bigl\{[a_n]:(\exists\epsilon>0)\,
                  \{n:\nu_n(a_n)>\epsilon n^2\}\in\UU\bigr\}.
                  \label{eq:Q}
\end{align}
Each quantified bound may hold on a different set in $\UU$; only
finite intersections are used below.

\begin{lemma}\label{lem:curves}
For $H_n=\langle n,n+1\rangle$ and $e=qn+r\geq0$, $0\leq r<n$,
$e\in H_n\Longleftrightarrow q\geq r$. Consequently
\begin{equation}\label{eq:curve-facts}
 A_n=\bigoplus_{r=0}^{n-1}t^{r(n+1)}\QQ[[t^n]]
 \cong\QQ[[X,Y]]/(Y^n-X^{n+1}),\qquad
 t^{n^2}\QQ[[t]]\subseteq A_n.
\end{equation}
\end{lemma}
\begin{proof}
If $q\geq r$, then $e=(q-r)n+r(n+1)$. Conversely, from
$e=an+b(n+1)$ and $b=kn+r$ one obtains $q=a+b+k\geq r$.
Grouping series supported on $H_n$ by their exponents modulo $n$
gives the direct sum. The conductor inclusion follows from
$e\geq n(n-1)\Rightarrow q\geq n-1\geq r$.

For the presentation, set $f=Y^n-X^{n+1}$. Modulo $f$, every
$F=\sum_{b\geq0}f_b(X)Y^b$ has the remainder
\[
 \sum_{r<n}g_r(X)Y^r,\qquad
 g_r(X)=\sum_{k\geq0}f_{kn+r}(X)X^{k(n+1)}.
\]
For $k\geq1$, every monomial of $(Y^{kn}-X^{k(n+1)})/f$ has total
degree at least $n(k-1)$, so the quotient series and displayed sums
converge formally. Substitution sends distinct remainder terms to distinct
exponent classes modulo $n$. A zero image therefore forces every
$g_r=0$, proving that the kernel is $(f)$.
\end{proof}

\begin{lemma}\label{lem:uniform}
For every $n\geq2$ and $r\geq1$, the relation module of any $r$
elements of $A_n$ has at most $r$ generators. Hence $A$ is coherent.
\end{lemma}
\begin{proof}
Write $\operatorname{Rel}_D(a_i)=\ker((a_i):D^r\to D)$.
Set $S=\QQ[[X,Y]]$, $f=Y^n-X^{n+1}$, and
$I=(a_1,\ldots,a_r)A_n$. If $I=0$, the relation module is $A_n^r$.
Otherwise let $K=\ker(S^r\twoheadrightarrow I)$.
We use standard power-series, DVR, and Nakayama facts
\cite[Tags~0306, 0AUW, 00DV]{StacksProject}.
Since $X$ acts on $I$ as the nonzero element $t^n$,
\[
 K\cap XS^r=XK,\qquad K/XK\hookrightarrow\QQ[[Y]]^r.
\]
This submodule is free of rank $\ell\leq r$. Lifting a basis gives
$\psi:S^\ell\to K$, surjective by Noetherianity and Nakayama.
For $L=\ker\psi$, the isomorphism modulo $X$ gives $L\subseteq XS^\ell$.
If $Xw\in L$, injectivity of $X$ on $K\subseteq S^r$ gives $w\in L$.
Thus $L=XL$, so Noetherianity and Nakayama give $L=0$.
Since $f\ne0$ annihilates $I$, tensoring with $\Frac S$ gives
$\ell=r$. Therefore
\[
 \operatorname{Rel}_{A_n}(a_1,\ldots,a_r)=K/fS^r
\]
is generated by the images of an $S$-basis of $K$.

For a row of length $r$ over $A$, choose $r$ relation generators in
each coordinate, padding by zeros. Every relation over $A$ holds
coordinatewise on a set in $\UU$, where it is a linear combination
of these generators; extend the coefficients by zero elsewhere.
Their classes therefore generate the relation module over $A$.
A ring is coherent exactly when every finite row has a finitely
generated relation module
\cite[\S10.90]{StacksProject}. Hence $A$ is coherent.
\end{proof}

\begin{proposition}\label{prop:ring}
The ideals $P\subsetneq Q$ are prime. The ring
\[
 R=(A/P)_{Q/P}
\]
is a one-dimensional coherent local domain, with maximal ideal
$\mm=(Q/P)R$. Writing bars for images in $R$, the element
$u=[t^{n^2}]$ satisfies $0\ne\bar u\in\mm$.
\end{proposition}
\begin{proof}
The order inequalities make $P,Q$ proper ideals. If $a,b\notin P$,
choose respective order bounds $Cn^2,Dn^2$ almost everywhere;
then $\nu_n(a_nb_n)\leq(C+D)n^2$ almost everywhere, so $ab\notin P$.
If $a,b\notin Q$, for each fixed $\epsilon>0$ both orders are at most
$\epsilon n^2/2$ almost everywhere, giving $ab\notin Q$.
Thus $P,Q$ are prime, with $P\subsetneq Q$ witnessed by
$u\in Q\setminus P$ and \eqref{eq:curve-facts}.

For $a\notin P$ we claim that
\begin{equation}\label{eq:divisible}
 aP=P.
\end{equation}
Choose $C\geq1$ with $\nu_n(a_n)\leq Cn^2$ almost everywhere.
For $p\in P$, define $b_n=p_n/a_n$ where this bound holds and
$\nu_n(p_n)>(C+1)n^2$, and set $b_n=0$ elsewhere.
The conductor inclusion puts $b_n$ in $A_n$, and $ab=p$.
For every fixed integer $D\geq1$, intersecting the bound on $a_n$
with $\nu_n(p_n)>(C+D)n^2$ gives $\nu_n(b_n)>Dn^2$ almost everywhere.
Thus this single $b$ lies in $P$, proving \eqref{eq:divisible}.

For a nonzero finite row over $A/P$, choose a lift $(a_i)$ with
$a_j\notin P$. Any lifted relation $(b_i)$ has error
$p=\sum_i a_i b_i\in P$. By \eqref{eq:divisible}, write $p=a_jc$
with $c\in P$; replacing $b_j$ by $b_j-c$ gives an exact relation
with the same reduction. Hence
\[
 \operatorname{Rel}_A(a_i)\twoheadrightarrow
 \operatorname{Rel}_{A/P}(a_i+P).
\]
The relation criterion gives coherence of $A/P$ (the zero row is
immediate). Exactness of localization gives coherence of $R$.
Primality makes $R$ a local domain with the stated maximal ideal.

For $a\in Q\setminus P$, we next show
\begin{equation}\label{eq:radical}
 \sqrt{aA}=Q.
\end{equation}
Choose $C$ as above. For $b\in Q$, choose $\epsilon>0$ with
$\nu_n(b_n)>\epsilon n^2$ almost everywhere, and then an integer
$N\geq1$ with $N\epsilon>C+1$.
Then $b_n^N/a_n$ has order greater than $n^2$ on the common set.
The conductor inclusion, with zero coordinates elsewhere, yields
$b^N\in aA$. The reverse inclusion follows from primality of $Q$.
If $P\subsetneq\mathfrak p\subseteq Q$ is prime, choose
$a\in\mathfrak p\setminus P$; \eqref{eq:radical} forces
$\mathfrak p=Q$. Prime correspondence gives
$\Spec R=\{0,\mm\}$, and $0\ne\bar u\in\mm$ proves $\dim R=1$.
\end{proof}

\section{The obstruction to Pr\"ufer integral closure}\label{sec:obstruction}

Set $F=\Frac R$ and
\[
 \alpha=\sqrt2-1,\qquad d_n=\lfloor\alpha n\rfloor,\qquad
 v=[t^{n^2+d_n}],\qquad z=\bar v/\bar u\in F.
\]
The conductor gives $u,v\in A$, and their orders are at most $2n^2$,
so $u,v\notin P$. Thus $z\in F$, although $t^{d_n}$ need not lie in
$A_n$. Irrationality will separate the low-order terms of a relation
for $z$.

\begin{lemma}\label{lem:separation}
Every polynomial in $R[T]$ vanishing at $z$ has all coefficients in
$\mm$. Consequently
\[
 R[z]/\mm R[z]\cong(R/\mm)[T],\qquad \dim R[z]\geq2.
\]
\end{lemma}
\begin{proof}
Suppose $\sum_{i=0}^d c_i z^i=0$ and some $c_i\notin\mm$;
necessarily $d\geq1$. Clear a common denominator to write
$c_i=\bar b_i/\bar s$, with $b_i\in A$ and $s\notin Q$.
Then $J=\{i:b_i\notin Q\}\ne\varnothing$ and, since
$A/P\hookrightarrow F$,
\begin{equation}\label{eq:error}
 h=\sum_{i=0}^d b_i v^i u^{d-i}\in P.
\end{equation}
For this fixed degree, irrationality gives
\[
 \delta=\min_{1\leq k\leq d}\dist(k\alpha,\mathbb Z)>0.
\]
Choose a common $\epsilon>0$ with
$\nu_n(b_{i,n})>\epsilon n^2$ almost everywhere for all $i\notin J$
(take $\epsilon=1$ if there are none), and fix
$0<\eta<\min\{\epsilon/4,\delta/4,1/4\}$.
On a common set $E\in\UU$ these lower bounds hold,
$\nu_n(b_{i,n})\leq\eta n^2$ for every $i\in J$, and
$d/n<\min\{\delta/4,\eta\}$.

For $n\in E$ and $i\in J$, write the leading exponent as
$e_{i,n}=q_{i,n}n+r_{i,n}$, $0\leq r_{i,n}<n$.
The semigroup criterion gives
\[
 0\leq r_{i,n}/n\leq q_{i,n}/n\leq e_{i,n}/n^2\leq\eta.
\]
The shifted orders $e_{i,n}+id_n$ are distinct for $i\in J$.
Otherwise equality for $i\ne j$, using
$q_{i,n}-q_{j,n}\in\mathbb Z$, would give
\[
 \delta\leq\dist((i-j)\alpha,\mathbb Z)
 \leq\frac{|r_{i,n}-r_{j,n}|}{n}
      +|i-j|\left|\alpha-\frac{d_n}{n}\right|
 <\eta+\frac dn<\frac\delta2.
\]
For $i\in J$, the order of $b_{i,n}t^{id_n}$ is less than
$2\eta n^2<\epsilon n^2$; for $i\notin J$ it exceeds
$\epsilon n^2$. Thus $g_n=\sum_i b_{i,n}t^{id_n}$ has a unique
summand of least order, and $\nu_n(g_n)<2\eta n^2<n^2$.
Since $h_n=t^{dn^2}g_n$, this gives
$\nu_n(h_n)<(d+1)n^2$ on $E$, contradicting \eqref{eq:error}.

The evaluation kernel lies in $\mm R[T]$, giving the quotient
isomorphism. The primes $(0)\subsetneq(T)$ of $(R/\mm)[T]$ yield
the prime chain
\[
 0\subsetneq\mm R[z]\subsetneq\mm R[z]+zR[z].
\]
The first inclusion is strict since $0\ne\bar u\in\mm R[z]$;
hence $\dim R[z]\geq2$.
\end{proof}

\begin{lemma}\label{lem:prufer-overrings}
Every overring of a one-dimensional Pr\"ufer domain has dimension
at most one.
\end{lemma}
\begin{proof}
Let $D\subseteq E\subseteq\Frac D$, with $D$ Pr\"ufer. For
$\mathfrak q\in\Spec E$ and $\mathfrak p=\mathfrak q\cap D$,
the local ring $E_{\mathfrak q}$ dominates $D_{\mathfrak p}$.
The latter is a valuation domain: its nonzero finitely generated
ideals are invertible, hence principal, and a local B\'ezout domain
is a valuation domain. If $x\in E_{\mathfrak q}\setminus D_{\mathfrak p}$,
then $x^{-1}\in\mathfrak pD_{\mathfrak p}\subseteq
\mathfrak qE_{\mathfrak q}$, contradicting $x x^{-1}=1$.
Thus $E_{\mathfrak q}=D_{\mathfrak p}$, and
$\dim E=\sup_{\mathfrak q}\dim E_{\mathfrak q}\leq\dim D$.
\end{proof}

\begin{proof}[Proof of \cref{thm:main}]
By \cref{prop:ring}, $R$ is a one-dimensional coherent local domain.
Let $B$ be its integral closure in $F$; then $\Frac B=F$.
Integral extensions preserve dimension \cite[Tag~00OK]{StacksProject},
so $\dim B=1$.
If $B$ were Pr\"ufer, \cref{lem:prufer-overrings} would give
$\dim B[z]\leq1$. Since $B[z]$ is generated over $R[z]$ by elements
integral over $R$, the extension $R[z]\subseteq B[z]$ is integral.
Dimension preservation and \cref{lem:separation} give
\[
 2\leq\dim R[z]=\dim B[z]\leq1,
\]
a contradiction.
\end{proof}

\section{Positive criteria and the obstruction}\label{sec:literature}

The positive results concern finiteness in the base ring, preservation
of coherence in extensions, and ideal-theoretic conditions on the
integral closure. We compare these approaches and identify the obstruction
supplied by our example. Throughout, $D$ is a one-dimensional coherent
domain, $K=\Frac D$, and $D'$ is its integral closure in $K$, unless weaker
hypotheses are stated. An overring lies between $D$ and $K$.

\subsection{Finiteness and extension properties}
Krull--Akizuki gives more than Pr\"ufer integral closure in the Noetherian
case: $D'$ is Dedekind. Even the Mori and $H$ conditions force a
one-dimensional coherent domain to be Noetherian
\cite[p.~225]{GlazVasconcelos1984}. Weakening Noetherianity in these
directions therefore does not reach the general problem. A different
approach asks whether coherence passes to $D'$. In dimension one, it
is enough that $D'$ be locally finite-conductor; $D'$ module-finite
over $D$ is one sufficient condition
\cite[Thm.~3.7, Cor.~3.9]{PicozzaTartarone2008}.
Finite conductor means that $aD\cap bD$ is finitely generated for all
$a,b\in D$; coherence implies this by applying the relation criterion
to $(a,-b)$.

For a nonzero ideal $J$ of $D$, put
\[
 J^{-1}=(D:J),\qquad J_v=(J^{-1})^{-1},\qquad
 J_t=\bigcup_{\substack{L\subseteq J\\L\text{ finitely generated}}}L_v
 \quad(J\ne0),
\]
with $0_v=0$. Divisorial means $J=J_v$; Mori means ACC on divisorial
ideals. The $H$ condition \cite[\S3]{GlazVasconcelos1977} is
\[
 J^{-1}=D\ \Longrightarrow\
 \exists L\subseteq J\text{ finitely generated with }L^{-1}=D.
\]
A DW domain has every maximal ideal a $t$-ideal; this is automatic in
dimension one \cite[Prop.~2.9]{PicozzaTartarone2008}. Thus the normal
DW finite-conductor criterion applies to $D'$ whenever its localizations
are finite-conductor. The missing inheritance property is precise:
finiteness of relations in $D$ need not yield finiteness of
principal-ideal intersections in $D'$.

Every row of Table~\ref{tab:positive-results} implies that $D'$ is
Pr\"ufer; stronger conclusions and changes to the standing assumptions
are indicated. A \emph{very finite} extension has every intermediate
ring, including the endpoint, module-finite over the base; a
\emph{coherent pair} has every intermediate ring coherent. These
hypotheses control intermediate rings, whereas finite generation of
an endpoint alone does not.

\begingroup
\small
\setlength{\tabcolsep}{5pt}
\renewcommand{\arraystretch}{1.08}
\begin{longtable}{@{}>{\raggedright\arraybackslash}p{0.455\textwidth}
                    >{\raggedright\arraybackslash}p{0.51\textwidth}@{}}
\caption{Additional hypotheses ensuring Pr\"ufer integral closure.
Each alternative includes its stated conjunction.}
\label{tab:positive-results}\\
\toprule
Additional hypothesis & Result and scope\\
\midrule
\endfirsthead
\multicolumn{2}{l}{\textit{Table \thetable\ continued}}\\
\toprule
Additional hypothesis & Result and scope\\
\midrule
\endhead
\midrule
\multicolumn{2}{r}{\textit{Continued on the next page}}\\
\endfoot
\bottomrule
\endlastfoot
Noetherian; Mori; $H$-domain
& $D$ Noetherian, $D'$ Dedekind
\cite[Tag~09IG]{StacksProject}; \cite[p.~225]{GlazVasconcelos1984}.\\
\addlinespace
Integrally closed; regular; finite weak or global dimension;
Pr\"ufer $v$-multiplication domain; $v$-domain
& $D$ Pr\"ufer
\cite[Cor.~11]{PapickCoherentOverrings1979};
\cite[Prop.~3.1]{TamekkanteAssaadBouba2019}.
Regular means that every finitely generated ideal has finite projective
dimension.\\
\addlinespace
$D'$ module-finite over $D$; $D'$ locally finite-conductor
& \cite[Thm.~3.7, Cor.~3.9]{PicozzaTartarone2008}.
The second includes locally coherent or coherent $D'$ and needs no
coherence assumption on $D$.\\
\addlinespace
Each nonnormal $D_{\mathfrak p}$ has a proper very finite overring
& Each such $D_{\mathfrak p}$ is Noetherian
\cite[Thm.~15, Cor.~16]{Papick1978}.\\
\addlinespace
$D$ local, with a proper very finite domain extension or a proper
integral coherent pair
& $D$ Noetherian
\cite[Thm.~2]{HuckabaPapick1981}; \cite[Rem.~9]{PapickCoherentOverrings1979}.
The extension need not lie in $K$.\\
\addlinespace
Every proper overring coherent; every local overring coherent
& \cite[Thm.~1, Rem.~2]{PapickProperOverrings1979}.
Neither dimension one nor coherence of $D$ is required.\\
\addlinespace
Locally a pseudo-valuation domain
& \cite[Cor.~3.6]{DobbsPVD1978} and localization.
A pseudo-valuation domain has a valuation overring sharing its maximal
ideal.\\
\addlinespace
A canonical module and finitely generated uppers to zero
& \cite[Prop.~4.14]{GlazVasconcelos1984}.
Coherence suffices; dimension one is unnecessary. Evaluation and
endomorphism formulations are discussed below.\\
\addlinespace
Finitely stable; stable; quasi-stable
& \cite[Prop.~3.4, Cor.~3.5]{PicozzaTartarone2010}.
No dimension or coherence hypothesis is needed.\\
\addlinespace
Clifford regular
& $D'$ Pr\"ufer of finite character
\cite[Prop.~3.4, Cor.~4.8]{Bazzoni2011}; no standing assumptions needed.\\
\addlinespace
Totally divisorial; Warfield
& $D$ stable and, in dimension one, Noetherian
\cite[pp.~263--264, Lemmas~3.2--3.3]{Olberding2001}.\\
\addlinespace
Finite expansion over a field $k$: integral over a finitely generated
$k$-subalgebra
& Apply \cite[Thm.~3.5]{Steiner2026} to $D'$.
Integrality preserves dimension and finite expansion; coherence is
unnecessary.\\
\end{longtable}
\endgroup

Since $B$ is not Pr\"ufer, none of the full sufficient hypotheses in
the table holds for our example. In particular, $R$ is non-Noetherian
and nonnormal, $B$ is not module-finite, and some $B_{\mathfrak n}$ is
not finite-conductor. This last failure is sharp: in a one-dimensional
normal local domain, intersections of incomparable principal ideals
are not finitely generated \cite[Rem.~2.2]{GuerrieriLoper2021}.

\subsection{Stability and polynomial relations}
An ideal $J\ne0$ is \emph{stable} if invertible over $(J:J)$.
Finite stability tests finitely generated ideals, stability tests all
nonzero ideals, and quasi-stability replaces invertibility by flatness.
These give a route to Pr\"ufer integral closure through endomorphism rings.
Clifford regularity means that the semigroup of nonzero fractional ideals
modulo principal ideals is von Neumann regular; it implies finite
stability. Total divisoriality requires every ideal of every overring to
be divisorial; Warfield duality implies this condition
\cite{Bazzoni2011,Olberding2001}. Divisoriality of finitely generated
ideals already forces a domain to be Pr\"ufer when normality is added
\cite[Cor.~3.2]{TamekkanteAssaadBouba2019}.

Polynomial relations connect these criteria to the construction.
An \emph{upper to zero} is a nonzero prime of $D[T]$ contracting to zero.
Write $c(f)$ for the ideal of coefficients of $f$; $f$ is primitive if
$c(f)=D$. For arbitrary domains, the characterizations collected in
\cite[Thm.~1.1]{ChangFontana2009} give
\[
\begin{aligned}
 D'\text{ Pr\"ufer}
 &\ \Longleftrightarrow\ D\text{ quasi-Pr\"ufer}\\
 &\ \Longleftrightarrow\ \text{every upper to zero contains a primitive polynomial}\\
 &\ \Longleftrightarrow\ \text{every }x\in K\text{ satisfies a primitive equation over }D\\
 &\ \Longleftrightarrow\ \text{every overring satisfies INC over }D.
\end{aligned}
\]
Here INC means that comparable primes have distinct contractions.
Equivalent conditions are that every overring is quasi-Pr\"ufer, or that
every prime of the Nagata ring $D[T]_N$ is extended from $D$, where $N$
consists of primitive polynomials. In dimension one, these are also
exactly the UMT domains, in which every upper to zero is a maximal
$t$-ideal; in arbitrary dimension, $D'$ is Pr\"ufer exactly when every
overring is UMT \cite[Thms.~1.1, 1.5, Cor.~3.11]{FontanaGabelliHouston1998}.

Our single evaluation upper displays these failures simultaneously.
With $\mathcal H=\ker(R[T]\to R[z])$, $T\mapsto z$,
\cref{lem:separation} gives
\begin{equation}\label{eq:literature-obstruction}
 0\ne\bar uT-\bar v\in\mathcal H\subsetneq\mm[T],\qquad
 R[z]/\mm R[z]\cong(R/\mm)[T].
\end{equation}
Thus $z$ has no primitive equation, and $R[z]$ has comparable primes
above $\mm$. Its fiber has a transcendental residue-field extension,
the obstruction to residual algebraicity studied in
\cite[Thm.~2.3, Cor.~2.8]{AyacheJaballah1997}.
Since $\mm$ is a $t$-ideal, this upper excludes UMT by
\cite[Thm.~1.1]{FontanaGabelliHouston1998}. It also yields
\[
 0\subsetneq\mathcal H\subsetneq\mm[T]
   \subsetneq\mm[T]+TR[T],\qquad \dim R[T]\geq3.
\]
For comparison, a one-dimensional domain with $\dim D[T]=2$ is
quasi-Pr\"ufer: failure produces an upper inside an extended maximal
ideal and hence a chain of length three. In particular, $R$ is not
Jaffard, since that property requires $\dim R[T]=\dim R+1$.

Glaz--Vasconcelos also link finite generation of relations to
stabilization. For a coherent $D$, $J=(a,b)$ with $b\ne0$, and
$\mathcal H_{a/b}=\ker(D[T]\to D[a/b])$, their
Propositions~4.10--4.11 give \cite{GlazVasconcelos1984}
\begin{align*}
 D[JT]\text{ coherent}
 &\ \Longrightarrow\ (J^n:J^n)\text{ eventually constant}\\
 &\ \Longleftrightarrow\ \mathcal H_{a/b}\text{ finitely generated}.
\end{align*}
A canonical module supplies duality on finitely generated torsion-free
modules. With this hypothesis, finite generation of all evaluation
uppers suffices in the proof of their Proposition~4.14. Their
Proposition~4.13 gives the exact ideal-theoretic criterion
\[
 D'\text{ Pr\"ufer}\quad\Longleftrightarrow\quad
 \forall J=(a,b)\ne0\ \exists n\geq1:
 J^n\text{ invertible over }(J^n:J^n).
\]
A failed conjunction does not exclude each factor: the
present argument leaves existence of a canonical module for $R$,
finite generation of $\mathcal H$, and coherence of $R[T]$ and $R[z]$
undetermined. The canonical-ideal claim reported as a private
communication in \cite[p.~336]{PapickCoherentOverrings1979} is not used
as a theorem.

\renewcommand{\sectionname}{}
\section*{Acknowledgments}
We used GPT-5.6 Sol and an agentic harness built around GPT-5.6 Sol and Claude Fable 5 to assist with literature searches, hypothesis testing, the exploration and elimination of potential approaches, wording refinement, and manuscript proofreading. We thank Hieu M. Vu, Tho Tran Huu, Khoi M. N. Nguyen, Dung V. Nguyen, and Quang X. Nguyen for their assistance with hardware-related matters and for providing technical support in the use of the AI tools and agentic system. 

\bibliographystyle{plain}
\bibliography{references}
\end{document}